\documentclass[10pt]{amsart}

\usepackage{graphicx}
\usepackage{amssymb}
\usepackage{bbm}
\usepackage{amsmath,mathtools}
\usepackage{amsthm}
\usepackage{stmaryrd}

\newcommand{\C}{\mathbb {C}}

\newcommand{\Z}{\mathbb {Z}} 
\newcommand{\N}{\mathbb {N}}

\theoremstyle{definition}

\newtheorem*{thm}{Theorem}
\newtheorem{lem}{Lemma}
\newtheorem*{dfn}{Definition}
\newtheorem*{rmk}{Remarks}
\newtheorem*{que}{Questions}

\title{A Cantor Spectrum Diagonal in $\mathcal{O}_2$}

\author{Philipp Sibbel}
\address{Mathematisches Institut \\
Universit\"at M\"unster\\
Einsteinstr.\ 62, 48149 M\"unster\\ Germany}
\email{philipp.sibbel@uni-muenster.de}

\author{Wilhelm Winter}
\address{Mathematisches Institut\\
Universit\"at M\"unster\\ 
Einsteinstr.\ 62, 48149 M\"unster\\ Germany}
\email{wwinter@uni-muenster.de}

\subjclass[2020]{46L05}
\keywords{Cartan subalgebra, C$^*$-diagonal, Cuntz algebra}

\thanks{Funded by the Deutsche Forschungsgemeinschaft (DFG, German Research Foundation) under Germany’s Excellence Strategy – EXC 2044 – 390685587, Mathematics Münster – Dynamics – Geometry – Structure, the Deutsche Forschungsgemeinschaft (DFG, German Research Foundation) – Project-ID 427320536 – SFB 1442, and ERC Advanced Grant 834267 - AMAREC}

\begin{document}
\begin{abstract}
We prove the existence of a C$^*$-diagonal in the Cuntz algebra $\mathcal{O}_2$ with spectrum homeomorphic to the Cantor space.
\end{abstract}

\maketitle

\noindent
In 1986, Kumjian defined the notion of a C$^*$-diagonal in a C$^*$-algebra as an analogue to the notion of a Cartan subalgebra in von Neumann algebras (\cite{kumjian1986c}). He showed that diagonal C$^*$-pairs occur as groupoid C$^*$-algebras of twisted étale equivalence relations. In \cite{renault2008cartan}, Renault noted that the notion of a C$^*$-diagonal is sometimes too restrictive as it does not cover a large number of canonical maximal abelian subalgebras in important C$^*$-algebras. As an example, he mentioned the Cuntz algebras $\mathcal{O}_n$ and more generally graph algebras, which contain obvious regular maximal abelian subalgebras that are not C$^*$-diagonals. The difference is that these C$^*$-algebras can be constructed as groupoid C$^*$-algebras from groupoids which have some isotropy, more precisely, they are only topologically principal as opposed to principal. As a remedy, Renault defined the weaker notion of a Cartan subalgebra in a C$^*$-algebra.

\begin{dfn}
Let $A$ be a C$^*$-algebra. An abelian sub-C$^*$-algebra $D$ in $A$ is called a \textit{Cartan subalgebra} if
\begin{enumerate}
\item[(0)] $D$ contains an approximate unit of $A$,
\item[(1)] $D$ is maximal abelian in $A$,
\item[(2)] $D$ is \textit{regular}, i.e.\ the normalisers of $D$,
\begin{equation*}
\mathcal{N}_A(D) := \{n \in A \mid n^*Dn \subset D \text{ and } nDn^* \subset D\},
\end{equation*}
generate $A$ as a C$^*$-algebra, and
\item[(3)] there exists a faithful conditional expectation $\Psi$ of $A$ onto $D$.
\end{enumerate}
$D$ is called a C$^*$-\textit{diagonal} if, moreover,
\begin{enumerate}
\item[(4)]
the unique extension property of pure states is satisfied, that is, every pure state on $D$ has a unique extension to a (necessarily pure) state on $A$.
\end{enumerate}
\end{dfn}

Renault then showed the analogous result to Kumjian's theorem, that is, the reduced C$^*$-algebra of a second-countable, topologically principal, locally compact, Hausdorff, étale, twisted groupoid contains a canonical Cartan subalgebra corresponding to the unit space, and every Cartan subalgebra in a separable C$^*$-algebra uniquely arises in this fashion (\cite{renault2008cartan}). A Cartan subalgebra then is a C$^*$-diagonal if and only if the corresponding groupoid is principal.

In the following years, motivated by results in the theory of von Neumann algebras, Cartan subalgebras were studied intensively in the C$^*$-algebraic setting. Several existence and non-uniqueness results for Cartan subalgebras and C$^*$-diagonals were obtained (e.g.\ \cite{renault2008cartan}, \cite{barlak2017cartan}, \cite{li2019cartan}, \cite{li2020every}, \cite{li2022constructing}), where the interplay between the study of Cartan subalgebras and classification theory of C$^*$-algebras attracted particular attention. In \cite{li2020every}, Li completed the proof that every classifiable (in the sense of Elliott's programme; cf.\ \cite{winter2018structure}) simple C$^*$-algebra has a Cartan subalgebra and that a unital, separable, simple C$^*$-algebra with finite nuclear dimension contains a Cartan subalgebra if and only if it satisfies the Universal Coefficient Theorem (UCT) of Rosenberg and Schochet (also see \cite{barlak2017cartan}).\newline

In \cite{cuntz1977simple}, Cuntz introduced an important class of C$^*$-algebras, which he denoted by $\mathcal{O}_n$, for $2 \le n \le \infty$, and which were later called Cuntz algebras. For $n$ finite, $\mathcal{O}_n$ is the universal C$^*$-algebra generated by $n$ isometries $S_1, \ldots, S_n$ with pairwise orthogonal range projections which add up to the unit. The Cuntz algebra $\mathcal{O}_n$ then contains a subalgebra isomorphic to $M_{n^\infty}$, the UHF algebra of type $n^\infty$. It is generated by elements of the form $S_\mu S_\nu^*$ for finite sequences $\mu, \nu \in \{1,\ldots,n\}^k$ of the same length, where $S_\mu = S_{\mu_1}\ldots S_{\mu_k}$ for $\mu = (\mu_1,\ldots,\mu_k)$. It is well-known that the subalgebra $D_{n^\infty}$ generated by diagonal matrices in $M_{n^\infty}$ is a C$^*$-diagonal. The corresponding subalgebra in the Cuntz algebra $\mathcal{O}_n$, which we also denote by $D_{n^\infty}$, is generated by elements of the form $S_\mu S_\mu^*$ and turns out to be a Cartan subalgebra in $\mathcal{O}_n$; however, the unique extension property of pure states is not satisfied and therefore it is not a C$^*$-diagonal in $\mathcal{O}_n$ (cf.\ \cite{renault2008cartan}). In particular, an element in $\{1,\ldots,n\}^\N$, the spectrum of $D_{n^\infty}$, corresponds to a pure state on $D_{n^\infty}$ with a unique extension to $\mathcal{O}_n$ if and only if the element is an (eventually) aperiodic sequence (\cite{cuntz1980automorphisms}, \cite{evans1980ono}).

The natural question whether $\mathcal{O}_n$ contains a C$^*$-diagonal is surprisingly difficult to answer. In \cite{hjelmborgstability}, Hjelmborg presented a construction of $\mathcal{O}_2$ as the C$^*$-algebra of a principal groupoid; by Renault's and Kumjian's work (\cite{renault2008cartan}, \cite{kumjian1986c}), this yields a C$^*$-diagonal in $\mathcal{O}_2$. Hjelmborg's argument factorises through Kirchberg--Phillips classification and, in particular, it uses Kirchberg's $\mathcal{O}_2$-absorption theorem, which states that $A \otimes \mathcal{O}_2$ is isomorphic to $\mathcal{O}_2$ if and only if $A$ is a simple, separable, unital and nuclear C$^*$-algebra. The spectrum of Hjelmborg's C$^*$-diagonal is $\mathbb{T} \times \{1,2\}^\N$, and in particular one-dimensional. Via tensoring with itself, one then also obtains C$^*$-diagonals with higher-dimensional spectra in $\mathcal{O}_2$.

Until now it remained unknown whether $\mathcal{O}_2$ contains a C$^*$-diagonal with Cantor spectrum. Below we will give an affirmative answer -- and here is why we are excited about this: A long-term goal is it to develop a structure theory for Cartan and diagonal sub-C$^*$-algebras of a given classifiable C$^*$-algebra. This programme is in its infancy, and it is natural to focus on spectra which are as universal as possible, and in particular robust under taking tensor products. Perhaps more important, we hope to eventually employ K-theoretic tools to build distinguishing invariants, and Cantor spectrum promises the most direct access to such methods.\newline

We write $M_n$ for the $n \times n$ - matrices with coefficients in $\C$, and $\mathcal{K} \cong \mathcal{K}(\ell^2(\N))$ for the compact operators on an infinite-dimensional separable Hilbert space. For $i,j \in \{1,\ldots,n\}$, let $e_{ij}$ denote the standard matrix units in $M_n$ (respectively in $\mathcal{K}$). For a non-unital C$^*$-algebra $A$, we write $A^+$ for its unitisation.\newline

Before giving full details, let us outline the construction. Let $X$ be a Cantor space and fix a minimal homeomorphism on $X$. Then, the associated $\Z$-action is free, which implies that $\mathcal{C}(X)$, the C$^*$-algebra of continuous functions on $X$, is a C$^*$-diagonal in the crossed product C$^*$-algebra $ \mathcal{C}(X) \rtimes_\alpha \Z$ (see \cite{renault2008cartan}), where $\alpha$ denotes the induced action on $\mathcal{C}(X)$. Let $v \in \mathcal{C}(X) \rtimes_\alpha \Z$ be the unitary implementing the action, i.e.\ $v h v^* = \alpha (h)$ for all $h \in \mathcal{C}(X)$.

Let $D_{2^\infty}$ denote the canonical Cartan subalgebra in the Cuntz algebra $\mathcal{O}_{2}$. It is generated by elements of the form $S_\mu S_\mu^*$ for $\mu\in \{1,2\}^k$ and $k \in \N$. Then,
\begin{equation*} 
D_{2^\infty} \otimes \mathcal{C}(X) \subset \mathcal{O}_2 \otimes (\mathcal{C}(X) \rtimes_\alpha \Z) 
\end{equation*}
is a Cartan subalgebra which is not a C$^*$-diagonal (see \cite[Lemma 5.1]{barlak2017cartan}). However, we will see that 
\begin{equation}\label{InclusionDinB}
D:= D_{2^\infty} \otimes \mathcal{C}(X) \subset \mathrm{C}^*(D_{2^\infty} \otimes \mathcal{C}(X), S_1 \otimes v, S_2 \otimes v)=:B 
\end{equation}
is a C$^*$-diagonal, and the ambient C$^*$-algebra $B$ is nuclear and simple; see Lemmas \ref{DiagonalInB} and \ref{BNuclear}.

But now, the infinite tensor product $D^{\otimes \infty}\subset B^{\otimes \infty}$ is still a C$^*$-diagonal with Cantor spectrum in a simple nuclear C$^*$-algebra. $B$ contains a unital copy of the Cuntz algebra $\mathcal{O}_2$, and since the latter is strongly self-absorbing (see \cite{toms2007strongly}), we have that $ B^{\otimes \infty} \cong B^{\otimes \infty} \otimes \mathcal{O}_2$; cf.\ \cite{rordam1994short}. The isomorphism with $\mathcal{O}_2$ then follows from Kirchberg's $\mathcal{O}_2$-absorption theorem, and we have our main result:

\begin{thm}\label{CantorDiagonalInO2}
$D^{\otimes \infty} $ is a C$^*$-diagonal in $B^{\otimes \infty} $ with spectrum homeomorphic to the Cantor space, and $B^{\otimes \infty} $ is isomorphic to the Cuntz algebra $\mathcal{O}_2$.
\end{thm}

We now carry out the steps above in detail.\newline

\begin{lem}\label{DiagonalInB}
$D \subset B$ as in \eqref{InclusionDinB} is a C$^*$-diagonal.
\end{lem}

\begin{proof}
Since $D_{2^\infty} \otimes \mathcal{C}(X) $ is a Cartan subalgebra in $ \mathcal{O}_2 \otimes (\mathcal{C}(X) \rtimes_\alpha \Z)$ by \cite[Lemma 5.1]{barlak2017cartan}, it is maximal abelian also in $B$ and there is a faithful conditional expectation $B \to D_{2^\infty} \otimes \mathcal{C}(X) $. To see that $D_{2^\infty}\otimes \mathcal{C}(X)$ is regular in $B$, it suffices to check that $S_1 \otimes v$ and $S_2\otimes v$ are normalisers. But for elementary tensors of the form $d \otimes f \in D_{2^\infty} \otimes \mathcal{C}(X)$, we have
\begin{equation*}
(S_i \otimes v) (d \otimes f)(S_i \otimes v)^* = (S_i d S_i^* \otimes vfv^*)
\end{equation*}
and
\begin{equation*}
(S_i \otimes v)^* (d \otimes f)(S_i \otimes v) = (S_i^* d S_i \otimes v^*fv),
\end{equation*}
which again lie in $D_{2^\infty}\otimes \mathcal{C}(X)$. Now, since sums of elementary tensors are dense in the closed subalgebra $D_{2^\infty} \otimes \mathcal{C}(X)$ and since conjugation with $S_i \otimes v$ and $(S_i \otimes v)^*$ is continuous, it follows that $S_i \otimes v$ are indeed normalisers.

Thus, it only remains to show the unique extension property of pure states. For this, consider an arbitrary pure state on $D_{2^\infty}\otimes \mathcal{C}(X)$. It is of the form $\rho \otimes \sigma$ for a pure state $\rho$ on $D_{2^\infty}$ and a pure state $\sigma$ on $\mathcal{C}(X)$. Let $\overline{\rho \otimes \sigma}$ be an extension to a pure state on $\mathcal{O}_2 \otimes ( \mathcal{C}(X) \rtimes_\alpha \Z)$. Since $\overline{\rho \otimes \sigma}( \mathbf{1}_{\mathcal{O}_2} \otimes \, . \,)$ is a state on $\mathcal{C}(X) \rtimes_\alpha \Z$ extending $\sigma$, it must be the unique extension, denoted by $\overline{\sigma}$, and hence pure. It follows from \cite[Lemma 11.3.6]{kadison1986fundamentals} that $\overline{\rho \otimes \sigma}=\overline{\rho} \otimes \overline{\sigma}$ for a state $\overline{\rho}$ on $\mathcal{O}_2$, which extends $\rho$. Note that a priori, $\overline{\rho}$ does not need to be unique on $\mathcal{O}_2$ but we claim that on $B$, $\overline{\rho} \otimes \overline{\sigma}$ already is uniquely determined. To see this, we consider an arbitrary product $k = (d \otimes f) k_1 \, \ldots\, k_l$ for $l \in \N$, where $d \in D_{2^\infty}$ and $f \in \mathcal{C}(X)$, and each $k_i$ is $S_1\otimes v$, $ S_1^*\otimes v^*$, $S_2\otimes v$ or $S_2^*\otimes v^*$. As the action on $X$ is free and $\overline{\sigma}$ is obtained using the conditional expectation of $ \mathcal{C}(X) \rtimes_\alpha \Z $ onto $ \mathcal{C}(X)$ followed by the point evaluation $\sigma$, $\overline{\rho}\otimes \overline{\sigma} \, (k)$ can only be nonzero if $v$ and $v^*$ occur the same number of times in the product. Therefore, if $\overline{\rho}\otimes \overline{\sigma} \, (k)$ is nonzero, then $k$ also contains the same number of $S_i$'s and $S_i^*$'s. Thus, the first tensor factor of $k$ has to be in the subalgebra of $\mathcal{O}_2$ generated by elements of the form $S_\mu S_\nu^*$ with $\ell(\mu) = \ell(\nu)$. On this subalgebra, which is isomorphic to the UHF algebra of type $2^\infty$, $\overline{\rho}$ also is uniquely determined. Now, since the linear span of elements like $k$ is dense in $B$, $\overline{\rho} \otimes \overline{\sigma}$ is uniquely determined on $B$.
\end{proof}

Our proof that $B$ is nuclear closely follows Cuntz's argument for $\mathcal{O}_n$. 
Our initial proof of the simplicity of $B$ also followed Cuntz's ideas, but discussions with S.\ Evington prompted us to use a result by Echterhoff (\cite[Corollary 2.7.35]{cuntz2017k}). In order to set up notation, we recall the construction in the case $n=2$ from \cite{cuntz1977simple}, where $\mathcal{O}_2$ was presented as a crossed product cut down by a projection.

We fix an isomorphism $ \phi : \mathcal{K} \to \mathcal{K} \otimes M_2$ such that $\phi(e_{11}) = e_{11} \otimes e_{11}$ and let $\text{id}_{M_{2^\infty}}: M_{2^\infty} \to M_{2^\infty}$ denote the identity map. Then, we can define an automorphism $\Phi : \mathcal{K} \otimes M_{2^\infty}\to \mathcal{K} \otimes M_{2^\infty}$ via $\Phi = \phi \otimes \text{id}_{M_{2^\infty}}$, that is 
\begin{equation*}
\Phi= \phi \otimes \text{id}_{M_{2^\infty}}: \mathcal{K} \otimes M_2 \otimes \ldots \longrightarrow (\mathcal{K} \otimes M_2) \otimes M_2 \otimes \ldots = \mathcal{K} \otimes M_{2^\infty}.
\end{equation*}
Now, $\Phi$ canonically extends to an automorphism $\Phi^+$ on the unitisation $(\mathcal{K} \otimes M_{2^\infty})^+$ and we obtain a $\Z$-action on $(\mathcal{K} \otimes M_{2^\infty})^+$. Thus, we can consider the crossed product $(\mathcal{K} \otimes M_{2^\infty})^+ \rtimes_{\Phi^+} \Z $. Let $u \in (\mathcal{K} \otimes M_{2^\infty})^+ \rtimes_{\Phi^+} \Z $ be the unitary implementing the action, i.e.\ $ux u^*= \Phi^+(x)$ for all $x \in (\mathcal{K} \otimes M_{2^\infty})^+$.

Now, $\mathcal{O}_2$ is (isomorphic to) the crossed product $(\mathcal{K} \otimes M_{2^\infty})^+\rtimes_{\Phi^+} \Z$ cut down by the projection $p: = e_{11} \otimes \mathbf{1}_{M_{2^\infty}} \in (\mathcal{K} \otimes M_{n^\infty})^+$, i.e.\
\begin{equation*}
\mathcal{O}_2 \cong p \, ((\mathcal{K} \otimes M_{2^\infty})^+\rtimes_{\Phi^+} \Z)\,p,
\end{equation*}
sending the generators $S_1$ and $S_2$ to $ up $ and $(e_{11} \otimes e_{21} \otimes \mathbf{1}_{M_2} \otimes \ldots ) \,up$, respectively (also see \cite[Section 4.2]{rordam2002classification}). Note that under this identification, we have
\begin{equation}\label{UHFviaGenerators}
{S}_{i_1} \ldots S_{i_k} {S}_{j_k}^* \ldots {S}_{j_1}^* = e_{11} \otimes e_{i_1 j_1} \otimes \ldots \otimes e_{i_k j_k} \otimes \mathbf{1}_{M_2} \otimes \ldots \, ,
\end{equation}
so $D_{2^\infty} \subset \mathcal{O}_2$ is identified with $e_{1 1} \otimes D_{2^\infty} \subset e_{11} \otimes M_{2^\infty} \subset p\,((\mathcal{K} \otimes M_{2^\infty})^+ \rtimes_{\Phi^+} \Z)\, p$.

We define elements 
\begin{align*}
\tilde{S}_1 &:= S_1 \otimes v = up \otimes v, \\
\tilde{S}_2 &:= S_2 \otimes v = (e_{11} \otimes e_{21} \otimes \mathbf{1}_{M_2} \otimes \ldots ) \, up \otimes v
\end{align*}
in $ \mathcal{O}_2 \otimes (\mathcal{C}(X) \rtimes_\alpha \Z)\subset((\mathcal{K} \otimes M_{2^\infty})^+\rtimes_{\Phi^+} \Z )\otimes (\mathcal{C}(X) \rtimes_\alpha \Z) $. Then, $\tilde{S}_1$ and $\tilde{S}_2$ also satisfy the Cuntz relations with unit $\tilde{p}:= p \otimes \mathbf{1}_{\mathcal{C}(X)} $, i.e.\ $\tilde{S}_i^* \tilde{S}_i = \tilde{S}_1 \tilde{S}_1^*+\tilde{S}_2\tilde{S}_2^* = \tilde{p}=\tilde{p}^*=\tilde{p}^2 $ and $\tilde{p} \tilde{S}_i \tilde{p}=\tilde{S}_i$. Furthermore, we have the identification
\begin{align*}
B &= \textnormal{C}^*(D_{2^\infty} \otimes \mathcal{C}(X), S_1 \otimes v , S_2 \otimes v) \\
 &\cong \textnormal{C}^*(e_{11}\otimes D_{2^\infty} \otimes \mathcal{C}(X), \tilde{S}_1 , \tilde{S}_2),
\end{align*}
when viewing $B$ as a subalgebra of $((\mathcal{K} \otimes M_{2^\infty})^+\rtimes_{\Phi^+} \Z )\otimes (\mathcal{C}(X) \rtimes_\alpha \Z) $.

\begin{lem}\label{BNuclear}
$B$ is nuclear and simple.
\end{lem}

\begin{proof}
First, we note that $\Phi^+ \otimes \alpha$ yields a $\Z$-action on $(\mathcal{K} \otimes M_{2^\infty})^+ \otimes \mathcal{C}(X)$ which is implemented by $u \otimes v$. Thus, by the universal property of the crossed product (cf.\ \cite[Definition II.10.3.7]{blackadar2006operator}), there exists a surjection 
\begin{equation*}
((\mathcal{K} \otimes M_{2^\infty})^+ \otimes \mathcal{C}(X)) \rtimes_{\Phi^+ \otimes \alpha} \Z \longrightarrow \mathrm{C}^*((\mathcal{K} \otimes M_{2^\infty})^+ \otimes \mathcal{C}(X), u \otimes v) =:E.
\end{equation*}
We consider $E$ as a subalgebra of $((\mathcal{K} \otimes M_{2^\infty})^+ \rtimes_{\Phi^+}\Z ) \otimes( \mathcal{C}(X) \rtimes_\alpha \Z)$.

As crossed products of nuclear C$^*$-algebras by amenable groups are nuclear (\cite[Proposition 14]{green1978local}), we find that $((\mathcal{K} \otimes M_{2^\infty})^+ \otimes \mathcal{C}(X)) \rtimes_{\Phi^+ \otimes \alpha} \Z $ is nuclear. Now, since $E$ is a quotient of a nuclear C$^*$-algebra, it is nuclear as well (\cite[Corollary 3.3]{choi1976separable}, \cite{tomiyama1970tensor}), and so is the hereditary sub-C$^*$-algebra $\tilde{p} E \tilde{p}$ (\cite[Corollary 3.3]{choi1978nuclear}). \newline

We now check that inside $((\mathcal{K}\otimes M_{2^\infty})^+ \rtimes_{\Phi^+} \Z) \otimes (\mathcal{C}(X) \rtimes_\alpha \Z)$, $B$ coincides with $ \tilde{p}E\tilde{p}$, which will conclude the proof of nuclearity. This will be done analogously to the presentation of $\mathcal{O}_2$ as a crossed product cut down by a projection (see \cite[Section 2.1]{cuntz1977simple}).

First, note that since $\tilde{S}_1 = \tilde{p} \tilde{S}_1\tilde{p}$ and $\tilde{S}_2 = \tilde{p} \tilde{S}_2\tilde{p}$, we have $B \subset \tilde{p} E\tilde{p}$. For the reverse inclusion, let us consider an $a \in E \subset ((\mathcal{K} \otimes M_{2^\infty})^+ \rtimes_{\Phi^+} \Z ) \otimes( \mathcal{C}(X) \rtimes_\alpha \Z)$ of the form
\begin{equation*}
a = \sum_{i=-N}^N x_i u^i \otimes y_i v^i,
\end{equation*}
where $N \in \N$ and $ x_i \in (\mathcal{K} \otimes M_{2^\infty})^+$, $y_i \in \mathcal{C}(X)$ for all $i= -N,\ldots,N$. Elements of this form are dense in $E$. When writing $\tilde{x}_i = u^{-i} x_i u^i $ and $\tilde{y}_i = v^{-i} y_i v^i $ for $i < 0$, we obtain 
\begin{equation*}
a = \sum_{i<0} u^i \tilde{x}_i \otimes v^i \tilde{y}_i+ x_0 \otimes y_0+ \sum_{i>0} x_i u^i \otimes y_i v^i.
\end{equation*}
Using that $up = S_1 = \mathbf{1}_{\mathcal{O}_2} S_1 = pup$, we see that $px_i u^i p = (px_ip)(up)^i $ for $i > 0$ and $pu^i\tilde{x}_i p = ((up)^*)^{-i} p \tilde{x}_i p$ for $i<0$. This yields
\begin{align*}
\tilde{p} a \tilde{p} &= \sum_{i<0} ((up)^*)^{-i} p \tilde{x}_ip \otimes v^i \tilde{y}_i+ px_0p \otimes v_0 + \sum_{i>0} (px_ip)(up)^i \otimes y_i v	^i\\
 &= \sum_{i<0} (S_1^*)^{-i} p \tilde{x}_ip \otimes v^i \tilde{y}_i+ px_0p\otimes y_0+ \sum_{i>0} (px_ip)S_1^i \otimes y_i v^i\\
 &= \sum_{i<0} ((S_1\otimes v)^*)^{-i} (p \tilde{x}_ip \otimes \tilde{y}_i)+ px_0p \otimes y_0+ \sum_{i>0} ((px_ip) \otimes y_i)(S_1 \otimes v)^i.
\end{align*}
Thus, $\tilde{p} E \tilde{p}$ is generated as a C$^*$-algebra by $p\, (\mathcal{K} \otimes M_{2^\infty})^+ \,p \otimes \mathcal{C}(X) = e_{11} \otimes M_{2^\infty} \otimes \mathcal{C}(X) $ together with $\tilde{S}_1$. Recall from \eqref{UHFviaGenerators} that 
\begin{equation*}
\tilde{S}_{i_1} \ldots \tilde{S}_{i_k} \tilde{S}_{j_k}^* \ldots \tilde{S}_{j_1}^* = e_{11} \otimes e_{i_1 j_1} \otimes \ldots \otimes e_{i_k j_k} \otimes \mathbf{1}_{M_2} \otimes \ldots \otimes \mathbf{1}_{\mathcal{C}(X)},
\end{equation*}
whence $e_{11} \otimes M_{2^\infty} \otimes \mathcal{C}(X) $ is generated by $\tilde{S}_1, \tilde{S}_2$ and $e_{11} \otimes \mathbf{1}_{M_{2^\infty}} \otimes \mathcal{C}(X)$. Therefore, $\tilde{p} E \tilde{p}$ is generated by $\tilde{S}_1$, $\tilde{S}_2$ and $e_{11} \otimes \mathbf{1}_{M_{2^\infty}} \otimes \mathcal{C}(X)$, which yields that $\tilde{p} E \tilde{p}$ is contained in $B$.\newline

We further observe that $\tilde{p}E\tilde{p} = \tilde{p} F \tilde{p}$, where $F = \mathrm{C}^*(\mathcal{K} \otimes M_{2^\infty} \otimes \mathcal{C}(X), u \otimes v)$. By the universal property of the crossed product, there exists a $^*$-homomorphism
\begin{equation*}
(\mathcal{K} \otimes M_{2^\infty} \otimes \mathcal{C}(X)) \rtimes_{\Phi \otimes \alpha} \Z \longrightarrow F,
\end{equation*}
the image of which contains $\tilde{p} F \tilde{p}$ as a hereditary subalgebra. By \cite[Corollary 2.7.35]{cuntz2017k}, the crossed product $(\mathcal{K} \otimes M_{2^\infty} \otimes \mathcal{C}(X)) \rtimes_{\Phi \otimes \alpha} \Z$ is simple, hence so is its image in $F$, and therefore also the hereditary subalgebra $\tilde{p} F \tilde{p}$, which in turn is isomorphic to $B$.
\end{proof}

We proceed by taking infinite tensor products and by \cite[Lemma 5.2]{barlak2017cartan}, we again obtain a C$^*$-diagonal
\begin{equation*}
D^{\otimes \infty} \subset B^{\otimes \infty}.
\end{equation*}
The spectrum of $D^{\otimes \infty} $ is the product of the individual spectra and therefore again a Cantor space.

Note that $ B^{\otimes \infty}$ is still simple, separable, unital and nuclear. Furthermore, $\mathcal{O}_2$ embeds unitally into increasingly higher tensor factors of $ B^{\otimes \infty}$, hence embeds unitally into the central sequence algebra $ \ell^\infty( \N, B^{\otimes \infty}) / c_0(\N , B^{\otimes \infty}) \cap (B^{\otimes \infty})^\prime$. This yields $B^{\otimes \infty} \cong B^{\otimes \infty} \otimes \mathcal{O}_2$ (\cite[p.\ 34]{rordam1994short}, \cite[Theorem 7.2.2]{rordam2002classification}), since $\mathcal{O}_2$ is strongly self-absorbing (cf.\ \cite{toms2007strongly}).

Now by Kirchberg's $\mathcal{O}_2$-absorption theorem (\cite[Corollary F]{kirchberg1994classification}, \cite[p.\ 952]{kirchberg1995exact}) and since $B^{\otimes \infty}$ is simple, separable, unital and nuclear, we have $B^{\otimes \infty} \cong B^{\otimes \infty} \otimes \mathcal{O}_2 \cong \mathcal{O}_2$; cf.\ \cite[Lemma 3.7]{kirchberg2000embedding}.

Hence, we have found a C$^*$-diagonal in $\mathcal{O}_2$ whose spectrum is a Cantor space and the proof of our theorem is complete. \hfill \qedsymbol \newline

We close with some remarks and open questions which are motivated by both our construction and our result.

\begin{rmk}
\begin{enumerate}
    \item[(i)] Algebras like our $B$ above already appeared in \cite{cuntz1981ktheoryII}, where a six-term exact sequence for their K-theory was derived. (With the notation of \cite{cuntz1981ktheoryII}, our $B$ can be written as $C(X) \times_{\mathcal{U}} \mathcal{O}_2$ with $\mathcal{U} =(u,u)$.)
    \item[(ii)] $B$ cannot coincide with $\mathcal{O}_2$ and it is therefore necessary to take the infinite tensor product $B^{\otimes \infty}$ in our construction. Indeed, S.\ Evington pointed out to us that in general $B$ carries interesting and non-trivial K-theoretic information; this will be pursued in upcoming work together with the first named author, where in particular, the K-theory of $B$ for certain Cantor minimal systems will be computed exactly.
\end{enumerate}
\end{rmk}

\begin{que}
\begin{enumerate}
\item[(i)] To what extent does our construction depend on the choice of the Cantor minimal system, and what would be a distinguishing invariant?
\item[(ii)] Do $(D \subset B) $ or $(D^{\otimes \infty} \subset B^{\otimes \infty})$ have finite diagonal dimension in the sense of \cite{li2023diagonal}?
\item[(iii)] Does the Cuntz algebra $\mathcal{O}_n$ for $n \ge 3$, or a general Kirchberg algebra contain a C$^*$-diagonal (with Cantor spectrum)? 
\end{enumerate}
\end{que}

We thank the referee for their careful reading of the manuscript and for pointing out the connection to Cuntz's paper \cite{cuntz1981ktheoryII}.

\bibliographystyle{plain}
\bibliography{References}

\bigskip
\end{document}